\documentclass[graybox]{svmult}

\usepackage{mathptmx}
\usepackage{helvet}
\usepackage{courier}
\usepackage{type1cm}

\usepackage{makeidx}

\usepackage{multicol}
\usepackage[bottom]{footmisc}

\usepackage{amsmath,amssymb}

\usepackage{cite}

\makeindex

\begin{document}

\title*{A symmetric N\"{o}rlund sum\\
with application to inequalities}

\author{Artur M. C. Brito da Cruz, Nat\'{a}lia Martins and Delfim F. M. Torres}

\authorrunning{A. M. C. Brito da Cruz, N. Martins and D. F. M. Torres}

\institute{Artur M. C. Brito da Cruz$^{1, 2}$\\
\email{artur.cruz@estsetubal.ips.pt}\\[0.2cm]
Nat\'{a}lia Martins$^{2}$\\
\email{natalia@ua.pt}\\[0.2cm]
Delfim F. M. Torres$^{2}$\\
\email{delfim@ua.pt}\\[0.3cm]
$^1$Escola Superior de Tecnologia de Set\'{u}bal, Estefanilha, 2910-761 Set\'{u}bal, Portugal\\[0.2cm]
$^2$Center for Research and Development in Mathematics and Applications\\
Department of Mathematics, University of Aveiro, 3810-193 Aveiro, Portugal}

\maketitle


\vspace*{-3cm}

{\small [Submitted 18-Oct-2011;
accepted 09-March-2012; to Proceedings of International
Conference on Differential \& Difference Equations and Applications,
in honour of Professor Ravi P. Agarwal, to be published by Springer
in the series \emph{Proceedings in Mathematics} (PROM)]}

\bigskip
\bigskip
\bigskip
\bigskip


\abstract{Properties of an
$\alpha,\beta$-symmetric N\"{o}rlund sum are studied.
Inspired in the work by Agarwal et al.,
$\alpha,\beta$-symmetric quantum versions of H\"{o}lder,
Cauchy--Schwarz and Minkowski inequalities are obtained.}


\section{Introduction}

The symmetric derivative of function $f$ at point $x$
is defined as $\lim_{h \rightarrow 0} (f(x+h)-f(x-h))/(2h)$.
The notion of symmetrically differentiable is interesting
because if a function is differentiable at a point then
it is also symmetrically differentiable, but the converse
is not true. The best known example of this fact
is the absolute value function: $f(x) = |x|$ is not
differentiable at $x = 0$ but is symmetrically differentiable
at $x = 0$ with symmetric derivative zero \cite{Thomsom}.

Quantum calculus is, roughly speaking, the equivalent to traditional
infinitesimal calculus but without limits \cite{Kac}. Therefore,
one can introduce the symmetric quantum derivative of $f$ at $x$
by $(f(x+h)-f(x-h))/(2h)$. As in any calculus, it is then natural
to develop a corresponding integration theory, looking to such integral
as the inverse operator of the derivative.

The main goal of this paper is to study
the properties of a general symmetric quantum integral
that we call, due to the so-called N\"{o}rlund sum \cite{Kac},
the $\alpha,\beta$-symmetric N\"{o}rlund sum.

The paper is organized as follows. In Section~\ref{sec:2}
we define the forward and backward N\"{o}rlund sums.
Then, in Section~\ref{sec:sns}, we introduce the
$\alpha,\beta$-symmetric N\"{o}rlund sum and give some
of its properties. We end with Section~\ref{sec:ineq},
proving $\alpha,\beta$-symmetric versions of H\"{o}lder's,
Cauchy--Schwarz's and Minkowski's inequalities.


\section{Forward and backward N\"{o}rlund sums}
\label{sec:2}

This section is dedicated to the inverse operators of
the $\alpha$-forward and $\beta$-backward differences,
$\alpha > 0$, $\beta > 0$, defined respectively by
\[
\Delta_{\alpha}\left[  f\right]  \left(  t\right)  :=\frac{f\left(
t+\alpha\right)  -f\left(  t\right)  }{\alpha}\, , \quad
\nabla_{\beta}\left[  f\right]  \left(  t\right)  :=\frac{f\left(  t\right)
-f\left(t-\beta\right)}{\beta}.
\]

\begin{definition}
\label{def:o1}
Let $I \subseteq \mathbb{R}$ be such that
$a,b\in I$ with $a<b$ and $\sup I=+\infty$.
For $f:I\rightarrow\mathbb{R}$ and $\alpha >0$ we define the N\"{o}rlund sum
(the $\alpha$-forward integral) of $f$ from $a$ to $b$ by
\[
\int_{a}^{b}f\left(  t\right)  \Delta_{\alpha}t=\int_{a}^{+\infty}f\left(
t\right)  \Delta_{\alpha}t-\int_{b}^{+\infty}f\left(  t\right)
\Delta_{\alpha}t,
\]
where
$\displaystyle \int_{x}^{+\infty}f\left(  t\right)  \Delta_{\alpha}t=\alpha\sum
_{k=0}^{+\infty}f\left(  x+k\alpha\right)$,
provided the series converges at $x=a$ and $x=b$. In that case, $f$ is
said to be $\alpha$-forward integrable on $\left[  a,b\right]$. We say that $f$
is $\alpha$-forward integrable over $I$ if it is $\alpha$-forward integrable
for all $a,b\in I$.
\end{definition}

Until Definition~\ref{beta} (the backward/nabla case),
we assume that $I$ is an interval of $\mathbb{R}$
such that $\sup I=+\infty$. Note that if
$f:I\rightarrow\mathbb{R}$ is a function
such that $\sup I<+\infty$, then we can extend function
$f$ to $\tilde{f}:\tilde{I}\rightarrow\mathbb{R}$,
where $\tilde{I}$ is an interval with $\sup\tilde{I}=+\infty$,
in the following way: $\tilde{f}|_{I}=f$ and
$\tilde{f}|_{\tilde{I}\backslash I}=0$.

Using the techniques of Aldwoah in his Ph.D. thesis \cite{Aldwoah},
it can be proved that the $\alpha$-forward integral has the following properties:

\begin{theorem}
If $f,g: I \rightarrow\mathbb{R}$ are
$\alpha$-forward integrable on $[a,b]$,
$c\in\left[a,b\right]$, $k\in\mathbb{R}$, then
\begin{enumerate}
\item $\displaystyle\int_{a}^{a}f\left(  t\right)\Delta_{\alpha}t=0$;

\item $\displaystyle\int_{a}^{b}f\left(  t\right)\Delta_{\alpha}t
=\int_{a}^{c}f\left(  t\right)  \Delta_{\alpha}t+\int_{c}^{b}f\left(  t\right)
\Delta_{\alpha}t$, when the integrals exist;

\item $\displaystyle\int_{a}^{b}f\left(  t\right)\Delta_{\alpha}t
=-\int_{b}^{a}f\left(  t\right)  \Delta_{\alpha}t$;

\item $kf$ is $\alpha$-forward integrable on $\left[a,b\right]$ and
$\displaystyle \int_{a}^{b}kf\left(t\right)\Delta_{\alpha}t
=k\int_{a}^{b}f\left(t\right)  \Delta_{\alpha}t$;

\item $f+g$ is $\alpha$-forward integrable on $\left[a,b\right]$ and
\[
\int_{a}^{b}\left(  f+g\right)  \left(  t\right)  \Delta_{\alpha}t=\int
_{a}^{b}f\left(  t\right)  \Delta_{\alpha}t+\int_{a}^{b}g\left(  t\right)
\Delta_{\alpha}t\text{;}
\]

\item if $f\equiv0$, then $\displaystyle\int_{a}^{b}f\left(  t\right)
\Delta_{\alpha}t=0$.
\end{enumerate}
\end{theorem}

\begin{theorem}
Let $f: I \rightarrow\mathbb{R}$ be $\alpha$-forward integrable
on $\left[a,b\right]$.
If $g:I\rightarrow\mathbb{R}$ is a nonnegative
$\alpha$-forward integrable function on $\left[a,b\right]$,
then $fg$ is $\alpha$-forward integrable on $\left[a,b\right]$.
\end{theorem}

\begin{proof}
Since $g$ is $\alpha$-forward integrable, then both series
$\alpha\sum_{k=0}^{+\infty}g\left(  a+k\alpha\right)$
and $\alpha\sum_{k=0}^{+\infty}g\left(  b+k\alpha\right)$
converge. We want to study the nature of series
$\alpha\sum_{k=0}^{+\infty}fg\left(  a+k\alpha\right)$
and $\alpha\sum_{k=0}^{+\infty}fg\left(  b+k\alpha\right)$.
Since there exists an order $N\in\mathbb{N}$ such that
$\left\vert fg\left(  b+k\alpha\right)  \right\vert \leqslant g\left(
b+k\alpha\right)$  and $\left\vert fg\left(  a+k\alpha\right)
\right\vert \leqslant g\left(  a+k\alpha\right)$
for all $k>N$, then both
$\alpha\sum_{k=0}^{+\infty}fg\left(a+k\alpha\right)$
and $\alpha\sum_{k=0}^{+\infty}fg\left(  b+k\alpha\right)$
converge absolutely. The intended conclusion follows.
\end{proof}

\begin{theorem}
\label{p}
Let $f:I\rightarrow\mathbb{R}$ and $p>1$.
If $\left\vert f\right\vert $ is $\alpha$-forward integrable
on $\left[  a,b\right]$, then $\left\vert f\right\vert ^{p}$
is also $\alpha$-forward integrable on $\left[a,b\right]$.
\end{theorem}

\begin{proof}
There exists $N\in\mathbb{N}$ such that
$\left\vert f\left(b+k\alpha\right)\right\vert^{p}
\leqslant
\left\vert f\left(b+k\alpha\right) \right\vert$
and
$\left\vert f\left(a+k\alpha\right)\right\vert^{p}
\leqslant
\left\vert f\left(a+k\alpha\right) \right\vert$
for all $k>N$. Therefore, $\left\vert f\right\vert^{p}$
is $\alpha$-forward integrable on $\left[a,b\right]$.
\end{proof}

\begin{theorem}
\label{desigualdade}
Let $f,g:I\rightarrow\mathbb{R}$ be $\alpha$-forward integrable on $\left[  a,b\right]$.
If $\left\vert f\left(  t\right)  \right\vert \leqslant g\left(  t\right)$
for all $t\in\left\{  a+k\alpha:k\in\mathbb{N}_{0}\right\}$,
then for $b\in\left\{  a+k\alpha:k\in\mathbb{N}_{0}\right\}$ one has
\[
\left\vert \int_{a}^{b}f\left(  t\right)  \Delta_{\alpha}t\right\vert
\leqslant\int_{a}^{b}g\left(  t\right)  \Delta_{\alpha}t.
\]
\end{theorem}

\begin{proof}
Since $b\in\left\{  a+k\alpha:k\in\mathbb{N}_{0}\right\}$,
there exists $k_{1}$ such that $b=a+k_{1}\alpha$. Thus,
\begin{align*}
\left\vert \int_{a}^{b}f\left(  t\right)  \Delta_{\alpha}t\right\vert  &
=\left\vert \alpha\sum_{k=0}^{+\infty}f\left(  a+k\alpha\right)  -\alpha
\sum_{k=0}^{+\infty}f\left(  a+\left(  k_{1}+k\right)  \alpha\right)
\right\vert \\
&=\left\vert \alpha\sum_{k=0}^{+\infty}f\left(  a+k\alpha\right)  -\alpha
\sum_{k=k_{1}}^{+\infty}f\left(  a+k\alpha\right)  \right\vert
=\left\vert \alpha\sum_{k=0}^{k_{1}-1}f\left(  a+k\alpha\right)  \right\vert\\
&\leqslant \alpha\sum_{k=0}^{k_{1}-1}\left\vert f\left(  a+k\alpha\right)
\right\vert
\leqslant\alpha\sum_{k=0}^{k_{1}-1}g\left(  a+k\alpha\right) \\
& =\alpha\sum_{k=0}^{+\infty}g\left(  a+k\alpha\right)  -\alpha\sum_{k=k_{1}
}^{+\infty}g\left(  a+k\alpha\right)
=\int_{a}^{b}g\left(  t\right)  \Delta_{\alpha}t.
\end{align*}
\end{proof}

\begin{corollary}
\label{desigualdade2}
Let $f,g:I\rightarrow\mathbb{R}$ be
$\alpha$-forward integrable on $\left[a,b\right]$
with $b = a+k\alpha$ for some $k\in\mathbb{N}_{0}$.

\begin{enumerate}
\item If $f\left(  t\right)  \geqslant 0$
for all $t\in\left\{  a+k\alpha:k\in\mathbb{N}_{0}\right\}$,
then $\int_{a}^{b}f\left(  t\right)  \Delta_{\alpha}t \geqslant 0$.

\item If $g\left(  t\right)  \geqslant f$ $\left(  t\right)$ for all
$t\in\left\{  a+k\alpha:k\in\mathbb{N}_{0}\right\}$, then
$\int_{a}^{b}g\left(  t\right)  \Delta_{\alpha}t\geqslant\int_{a}^{b}f\left(
t\right)  \Delta_{\alpha}t$.
\end{enumerate}
\end{corollary}

We can now prove the following fundamental
theorem of the $\alpha$-forward calculus.

\begin{theorem}[Fundamental theorem of N\"{o}rlund calculus]
Let $f:I\rightarrow\mathbb{R}$ be $\alpha$-forward integrable
over $I$. Let $x\in I$ and define
$F\left(  x\right)  :=\int_{a}^{x}f\left(  t\right)  \Delta_{\alpha}t$.
Then, $\Delta_{\alpha}\left[  F\right]  \left(  x\right)  =f\left(  x\right)$.
Conversely,
$\int_{a}^{b}\Delta_{\alpha}\left[  f\right]  \left(  t\right)  \Delta_{\alpha}t
=f\left(  b\right)  -f\left(  a\right)$.
\end{theorem}

\begin{proof}
If $G\left(  x\right)=
-\int_{x}^{+\infty}f\left(t\right)\Delta_{\alpha}t$, then
\begin{align*}
\Delta_{\alpha}\left[  G\right]  \left(  x\right)
&=\frac{G\left(x+\alpha\right)  -G\left(  x\right)  }{\alpha}
=\frac{-\alpha\sum_{k=0}^{+\infty}f\left(  x+\alpha+k\alpha\right)
+\alpha\sum_{k=0}^{+\infty}f\left(  x+k\alpha\right)  }{\alpha}\\
& =\sum_{k=0}^{+\infty}f\left(  x+k\alpha\right)
-\sum_{k=0}^{+\infty}f\left(  x+\left(  k+1\right)  \alpha\right)
=f\left(  x\right).
\end{align*}
Therefore,
$\Delta_{\alpha}\left[  F\right]  \left(  x\right)
=\Delta_{\alpha}\left(\int_{a}^{+\infty}f\left(  t\right)
\Delta_{\alpha}t-\int_{x}^{+\infty}
f\left(  t\right)  \Delta_{\alpha}t\right)
=f\left(  x\right)$.
Using the definition of $\alpha$-forward difference operator, the second part
of the theorem is also a consequence of the properties of Mengoli's series.
Since
\begin{align*}
\int_{a}^{+\infty}\Delta_{\alpha}\left[  f\right]  \left(  t\right)
\Delta_{\alpha}t  & =\alpha\sum_{k=0}^{+\infty}\Delta_{\alpha}\left[
f\right]  \left(  a+k\alpha\right)
=\alpha\sum_{k=0}^{+\infty}\frac{f\left(  a+k\alpha+\alpha\right)  -f\left(
a+k\alpha\right)  }{\alpha}\\
& =\sum_{k=0}^{+\infty}\bigg(f\left(  a+\left(  k+1\right)  \alpha\right)
-f\left(  a+k\alpha\right)  \bigg) =-f\left(  a\right)
\end{align*}
and
$\int_{b}^{+\infty}\Delta_{\alpha}\left[  f\right]  \left(  t\right)
\Delta_{\alpha}t=-f\left(  b\right)$,
it follows that
\begin{equation*}
\int_{a}^{b}\Delta_{\alpha}\left[  f\right]  \left(  t\right)  \Delta_{\alpha}t
=\int_{a}^{+\infty}f\left(  t\right)  \Delta_{\alpha}t-\int_{b}^{+\infty}
f\left(  t\right)  \Delta_{\alpha}t
=f\left(  b\right)  -f\left(  a\right).
\end{equation*}
\end{proof}

\begin{corollary}[$\alpha$-forward integration by parts]
\label{partes}
Let $f,g:I\rightarrow\mathbb{R}$.
If $f g$ and $f\Delta_{\alpha}\left[  g\right]  $ are $\alpha
$-forward integrable on $\left[  a,b\right]  $, then
\[
\int_{a}^{b}f\left(  t\right)  \Delta_{\alpha}\left[  g\right]  \left(
t\right)  \Delta_{\alpha}t=f\left(  t\right)  g\left(  t\right)
\bigg|_{a}^{b}-\int_{a}^{b}\Delta_{\alpha}\left[  f\right]  \left(  t\right)
g\left(  t+\alpha\right)  \Delta_{\alpha}t
\]
\end{corollary}

\begin{proof}
Since
$\Delta_{\alpha}\left[  fg\right]  \left(  t\right)  =\Delta_{\alpha}\left[
f\right]  \left(  t\right)  g\left(  t+\alpha\right)  +f\left(  t\right)
\Delta_{\alpha}\left[  g\right]  \left(  t\right)$, then
\begin{align*}
\int_{a}^{b}f\left(  t\right)  \Delta_{\alpha}\left[  g\right]  \left(
t\right)  \Delta_{\alpha}t  & =\int_{a}^{b}\bigg(\Delta_{\alpha}\left[
fg\right]  \left(  t\right)  -\Delta_{\alpha}\left[  f\right]  \left(
t\right)  g\left(  t+\alpha\right)  \bigg)\Delta_{\alpha}t\\
& =\int_{a}^{b}\Delta_{\alpha}\left[  fg\right]  \left(  t\right)
\Delta_{\alpha}t-\int_{a}^{b}\Delta_{\alpha}\left[  f\right]  \left(
t\right)  g\left(  t+\alpha\right)  \Delta_{\alpha}t\\
& =f\left(  t\right)  g\left(  t\right)  \bigg|_{a}^{b}-\int_{a}^{b}
\Delta_{\alpha}\left[  f\right]  \left(  t\right)  g\left(  t+\alpha\right)
\Delta_{\alpha}t\text{.}
\end{align*}
\end{proof}

\begin{remark}
Our study of the N\"{o}rlund sum is in agreement with
the Hahn quantum calculus \cite{Aldwoah,withMiguel01,MalinowskaTorres}.
In \cite{Kac}
$\int_{a}^{b}f\left(  t\right)  \Delta_{\alpha}t=\alpha\left[  f\left(
a\right)  +f\left(  a+\alpha\right)
+ \cdots +f\left(  b-\alpha\right)  \right]$
for $a<b$ such that $b-a\in\alpha\mathbb{Z}$,
$\alpha\in\mathbb{R}^{+}$. In contrast with \cite{Kac}, our definition is valid
for any two real points $a,b$ and not only for those points belonging to the time
scale $\alpha\mathbb{Z}$. The definitions (only) coincide
if function $f$ is $\alpha$-forward
integrable on $\left[a,b\right]$.
\end{remark}

Similarly, we introduce the $\beta$-backward integral.

\begin{definition}
\label{beta}
Let $I$ be an interval of $\mathbb{R}$ such that $a,b\in I$
with $a<b$ and $\inf I=-\infty$. For $f:I\rightarrow\mathbb{R}$
and $\beta >0$ we define the $\beta$-backward integral of $f$ from $a$ to $b$ by
\[
\int_{a}^{b}f\left(  t\right)  \nabla_{\beta}t=\int_{-\infty}^{b}f\left(
t\right)  \nabla_{\beta}t-\int_{-\infty}^{a}f\left(  t\right)  \nabla_{\beta}t,
\]
where $\displaystyle \int_{-\infty}^{x}f\left(  t\right)\nabla_{\beta}t
=\beta\sum_{k=0}^{+\infty}f\left(  x-k\beta\right)$,
provided the series converges at $x=a$ and $x=b$. In that case, $f$ is
called $\beta$-backward integrable on $\left[a,b\right]$. We say that $f$
is $\beta$-backward integrable over $I$ if it is $\beta$-backward integrable
for all $a,b\in I$.
\end{definition}

The $\beta$-backward N\"{o}rlund sum has similar results and properties as the
$\alpha$-forward N\"{o}rlund sum. In particular, the $\beta$-backward integral
is the inverse operator of $\nabla_\beta$.


\section{The $\alpha,\beta$-symmetric N\"{o}rlund sum}
\label{sec:sns}

We define the $\alpha,\beta$-symmetric integral
as a linear combination of the $\alpha$-forward
and the $\beta$-backward integrals.

\begin{definition}
\label{def:3}
Let $f:\mathbb{R}\rightarrow\mathbb{R}$ and $a,b\in\mathbb{R}$, $a<b$.
If $f$ is $\alpha$-forward and $\beta$-backward integrable
on $\left[  a,b\right]$, $\alpha, \beta \ge 0$ with $\alpha + \beta > 0$,
then we define the $\alpha,\beta$-symmetric integral of $f$ from $a$ to $b$ by
\[
\int_{a}^{b}f\left(  t\right)  d_{\alpha,\beta}t=\frac{\alpha}{\alpha+\beta
}\int_{a}^{b}f\left(  t\right)  \Delta_{\alpha}t+\frac{\beta}{\alpha+\beta
}\int_{a}^{b}f\left(  t\right)  \nabla_{\beta}t\text{.}
\]
Function $f$ is $\alpha,\beta$-symmetric integrable if it is
$\alpha,\beta$-symmetric integrable for all $a,b\in\mathbb{R}$.
\end{definition}

\begin{remark}
Note that if $ \alpha\in\mathbb{R}^{+}$ and $\beta=0$,
then $\displaystyle\int_{a}^{b}f\left(  t\right)
d_{\alpha,\beta}t=\int_{a}^{b}f\left(  t\right)  \Delta_{\alpha}t$
and we do not need to assume in Definition~\ref{def:3} that
$f$ is $\beta$-backward integrable;
if $\alpha=0$ and $\beta\in\mathbb{R}^{+}$, then
$\displaystyle\int_{a}^{b}f\left(  t\right)  d_{\alpha,\beta}t
=\int_{a}^{b}f\left(  t\right)  \nabla_{\beta}t$
and we do not need to assume that
$f$ is $\alpha$-forward integrable.
\end{remark}

\begin{example}
Let $f\left(t\right)= 1/t^{2}$. Then
$\displaystyle \int_{1}^{3}\frac{1}{t^{2}}d_{2,2}t = \frac{10}{9}$.
\end{example}

The $\alpha,\beta$-symmetric integral has the following properties:

\begin{theorem}
\label{propriedades}
Let $f,g:\mathbb{R}\rightarrow\mathbb{R}$
be $\alpha,\beta$-symmetric integrable on $\left[a,b\right]$.
Let $c\in\left[  a,b\right]$ and $k\in\mathbb{R}$. Then,
\begin{enumerate}
\item $\displaystyle\int_{a}^{a}f\left(  t\right)  d_{\alpha,\beta}t=0$;

\item $\displaystyle\int_{a}^{b}f\left(  t\right)  d_{\alpha,\beta}t=\int
_{a}^{c}f\left(  t\right)  d_{\alpha,\beta}t+\int_{c}^{b}f\left(  t\right)
d_{\alpha,\beta}t$, when the integrals exist;

\item $\displaystyle\int_{a}^{b}f\left(  t\right)  d_{\alpha,\beta}t
=-\int_{b}^{a}f\left(  t\right)  d_{\alpha,\beta}t$;

\item $kf$ is $\alpha,\beta$-symmetric integrable on $\left[a,b\right]$
and
$\displaystyle \int_{a}^{b}kf\left(t\right) d_{\alpha,\beta}t
=k\int_{a}^{b}f\left(t\right) d_{\alpha,\beta}t$;

\item $f+g$ is $\alpha,\beta$-symmetric integrable
on $\left[a,b\right]$ and
\[
\int_{a}^{b}\left(  f+g\right)  \left(  t\right)  d_{\alpha,\beta}t=\int
_{a}^{b}f\left(  t\right)  d_{\alpha,\beta}t+\int_{a}^{b}g\left(  t\right)
d_{\alpha,\beta}t\text{;}
\]

\item $fg$ is $\alpha,\beta$-symmetric integrable on $\left[  a,b\right]$
provided $g$ is a nonnegative function.
\end{enumerate}
\end{theorem}

\begin{proof}
These results are easy consequences of the $\alpha$-forward
and $\beta$-backward integral properties.
\end{proof}

The next result follows immediately from Theorem~\ref{p} and the
corresponding $\beta$-backward version.

\begin{theorem}
\label{modulo}
Let $f:\mathbb{R}\rightarrow\mathbb{R}$ and $p>1$.
If $\left\vert f\right\vert $ is symmetric $\alpha,\beta$-integrable
on $\left[  a,b\right]$, then $\left\vert f\right\vert ^{p}$
is also $\alpha,\beta$-symmetric integrable on $\left[a,b\right]$.
\end{theorem}

\begin{theorem}
\label{des}
Let $f,g:\mathbb{R}\rightarrow\mathbb{R}$
be $\alpha,\beta$-symmetric integrable functions on $\left[a,b\right]$,
$\mathcal{A} := \left\{  a+k\alpha:k\in\mathbb{N}_{0}\right\}$
and $\mathcal{B} := \left\{  b-k\beta:k\in\mathbb{N}_{0}\right\}$.
For $b\in\mathcal{A}$ and $a\in\mathcal{B}$ one has:
\begin{enumerate}
\item if $\left\vert f\left(  t\right)  \right\vert
\leqslant g\left(t\right)$ for all $t\in\mathcal{A}\cup\mathcal{B}$, then
$\displaystyle \left\vert \int_{a}^{b}f\left(  t\right)  d_{\alpha,\beta}t\right\vert
\leqslant\int_{a}^{b}g\left(  t\right)  d_{\alpha,\beta}t$;

\item if $f\left(  t\right) \geqslant0$ for all
$t\in\mathcal{A}\cup\mathcal{B}$, then
$\displaystyle \int_{a}^{b}f\left(  t\right)  d_{\alpha,\beta}t\geqslant0$;

\item if $g\left(  t\right)  \geqslant f\left(  t\right)$
for all $t\in\mathcal{A}\cup\mathcal{B}$, then
$\displaystyle \int_{a}^{b}g\left(  t\right)
d_{\alpha,\beta}t\geqslant\int_{a}^{b}f\left(t\right)d_{\alpha,\beta}t$.
\end{enumerate}
\end{theorem}

\begin{proof}
It follows from Theorem~\ref{desigualdade} and Corollary~\ref{desigualdade2}
and the corresponding $\beta$-backward versions.
\end{proof}

In Theorem~\ref{thm:mvt} we assume that $a,b\in\mathbb{R}$ with
$b\in\mathcal{A} := \left\{  a+k\alpha:k\in\mathbb{N}_{0}\right\}$ and
$a\in\mathcal{B} := \left\{  b-k\beta:k\in\mathbb{N}_{0}\right\}$, where
$\alpha,\beta\in\mathbb{R}_{0}^{+}$, $\alpha + \beta \ne 0$.

\begin{theorem}[Mean value theorem]
\label{thm:mvt}
Let $f,g:\mathbb{R}\rightarrow\mathbb{R}$
be bounded and $\alpha,\beta$-symmetric integrable on $[a,b]$
with $g$ nonnegative. Let $m$ and $M$ be
the infimum and the supremum, respectively, of function $f$.
Then, there exists a real number $K$ satisfying the inequalities
$m\leqslant K\leqslant M$ such that
$\displaystyle \int_{a}^{b}f\left(  t\right)  g\left(  t\right)  d_{\alpha,\beta}t
=K\int_{a}^{b}g\left(  t\right)  d_{\alpha,\beta}t$.
\end{theorem}

\begin{proof}
Since $m\leqslant f\left(  t\right)  \leqslant M\text{ for all }t\in\mathbb{R}$
and $g\left(  t\right)  \geqslant 0$, then
$mg\left(  t\right)  \leqslant f\left(  t\right)  g\left(  t\right)  \leqslant
Mg\left(  t\right)$ for all $t\in\mathcal{A} \cup \mathcal{B}$.
All functions $mg$, $fg$ and $Mg$ are
$\alpha,\beta$-symmetric integrable on $[a,b]$.
By Theorems~\ref{propriedades} and \ref{des},
$m\int_{a}^{b}g\left(  t\right)  d_{\alpha,\beta}t
\leqslant\int_{a}^{b}f\left(t\right)  g\left(  t\right)  d_{\alpha,\beta}t$
$\leqslant M\int_{a}^{b}g\left(
t\right)  d_{\alpha,\beta}t$.
If $\int_{a}^{b}g\left(  t\right)  d_{\alpha,\beta}t=0$, then
$\int_{a}^{b}f\left(  t\right)  g\left(  t\right)
d_{\alpha,\beta}t=0$; if $\int_{a}^{b}g\left(  t\right)
d_{\alpha,\beta}t>0$, then
$m\leqslant\frac{\int_{a}^{b}f\left(  t\right)  g\left(  t\right)
d_{\alpha,\beta}t}{\int_{a}^{b}g\left(  t\right)
d_{\alpha,\beta}t} \leqslant M$.
Therefore, the middle term of these inequalities
is equal to a number $K$, which yields the intended result.
\end{proof}


\section{$\alpha,\beta$-Symmetric Integral Inequalities}
\label{sec:ineq}

Inspired in the work by Agarwal et al. \cite{Agarval},
we now present $\alpha,\beta$-symmetric versions of H\"{o}lder,
Cauchy--Schwarz and Minkowski inequalities. As before,
we assume that $a,b\in\mathbb{R}$ with
$b\in\mathcal{A} := \left\{  a+k\alpha:k\in\mathbb{N}_{0}\right\}$ and
$a\in\mathcal{B} := \left\{  b-k\beta:k\in\mathbb{N}_{0}\right\}$, where
$\alpha,\beta\in\mathbb{R}_{0}^{+}$, $\alpha + \beta \ne 0$.

\begin{theorem}[H\"{o}lder's inequality]
\label{Holders Inequality}
Let $f,g:\mathbb{R}\rightarrow\mathbb{R}$
and $a,b\in\mathbb{R}$ with $a<b$.
If $\left\vert f\right\vert $ and $\left\vert g\right\vert$ are
$\alpha,\beta$-symmetric integrable on $\left[a,b\right]$, then
\begin{equation}
\label{eq:hi}
\int_{a}^{b}\left\vert f\left(  t\right)  g\left(  t\right)  \right\vert
d_{\alpha,\beta}t\leqslant\left(  \int_{a}^{b}\left\vert f\left(  t\right)
\right\vert ^{p}d_{\alpha,\beta}t\right)  ^{\frac{1}{p}}\left(
\int_{a}^{b}\left\vert g\left(  t\right)\right\vert^{q}
d_{\alpha,\beta}t\right)^{\frac{1}{q}},
\end{equation}
where $p>1$ and $q=p/(p-1)$.
\end{theorem}

\begin{proof}
For $\alpha,\beta\in\mathbb{R}_{0}^{+}$, $\alpha + \beta \ne 0$,
the following inequality holds:
$\alpha^{\frac{1}{p}}\beta^{\frac{1}{q}}\leqslant\frac{\alpha}{p}+\frac{\beta}{q}$.
Without loss of generality, suppose that
$\displaystyle \left(  \int_{a}^{b}\left\vert f\left(  t\right)  \right\vert ^{p}
d_{\alpha,\beta}t\right)  \left(  \int_{a}^{b}\left\vert g\left(  t\right)
\right\vert ^{q}d_{\alpha,\beta}t\right)  \neq 0$
(note that both integrals exist by Theorem~\ref{modulo}). Set
$\xi\left( t\right)
=\left\vert f\left(  t\right)  \right\vert ^{p}/
\int_{a}^{b}\left\vert f\left(  \tau\right)  \right\vert ^{p}d_{\alpha,\beta}\tau$
and $\gamma\left(  t\right)
=\left\vert g\left(t\right)\right\vert ^{q}/\int_{a}^{b}\left\vert g\left(\tau\right)
\right\vert ^{q}d_{\alpha,\beta}\tau$.
Since both functions $\alpha$ and $\beta$ are symmetric
$\alpha,\beta$-integrable on $\left[a,b\right]$, then \eqref{eq:hi} holds:
\begin{align*}
\int_{a}^{b}&\frac{\left\vert f\left(  t\right)  \right\vert }{\left(
\int_{a}^{b}\left\vert f\left(  \tau\right)  \right\vert ^{p}d_{\alpha,\beta
}\tau\right)  ^{\frac{1}{p}}}\frac{\left\vert g\left(  t\right)  \right\vert
}{\left(  \int_{a}^{b}\left\vert g\left(  \tau\right)  \right\vert
^{q}d_{\alpha,\beta}\tau\right)  ^{\frac{1}{q}}}d_{\alpha,\beta}t
=\int_{a}^{b}\xi\left(  t\right)  ^{\frac{1}{p}}\gamma\left(  t\right)
^{\frac{1}{q}}d_{\alpha,\beta}t\\
& \leqslant\int_{a}^{b}\left(  \frac{\xi\left(  t\right)  }{p}+\frac
{\gamma\left(  t\right)  }{q}\right)  d_{\alpha,\beta}t\\
& =\frac{1}{p}\int_{a}^{b}\left(  \frac{\left\vert f\left(  t\right)
\right\vert ^{p}}{\int_{a}^{b}\left\vert f\left(  \tau\right)  \right\vert
^{p}d_{\alpha,\beta}\tau}\right)  d_{\alpha,\beta}t
+\frac{1}{q}\int_{a}^{b}\left(  \frac{\left\vert g\left(  t\right)
\right\vert ^{q}}{\int_{a}^{b}\left\vert g\left(  \tau\right)
\right\vert ^{q}d_{\alpha,\beta}\tau}\right)  d_{\alpha,\beta}t = 1.
\end{align*}
\end{proof}

The particular case $p=q=2$ of \eqref{eq:hi}
gives the Cauchy--Schwarz inequality.

\begin{corollary}[Cauchy--Schwarz's inequality]
Let $f,g:\mathbb{R}\rightarrow\mathbb{R}$
and $a,b\in\mathbb{R}$ with $a<b$. If $f$ and $g$
are $\alpha,\beta$-symmetric integrable on $\left[  a,b\right]$, then
\[
\int_{a}^{b}\left\vert f\left(  t\right)  g\left(  t\right)  \right\vert
d_{\alpha,\beta}t\leqslant\sqrt{\left(  \int_{a}^{b}\left\vert f\left(
t\right)  \right\vert ^{2}d_{\alpha,\beta}t\right)  \left(
\int_{a}^{b}\left\vert g\left(  t\right)
\right\vert ^{2}d_{\alpha,\beta}t\right)}\text{.}
\]
\end{corollary}

We prove the Minkowski inequality using H\"{o}lder's inequality.

\begin{theorem}[Minkowski's inequality]
Let $f,g:\mathbb{R}\rightarrow\mathbb{R}$
and $a,b,p\in\mathbb{R}$ with $a<b$ and $p>1$.
If $f$ and $g$ are $\alpha,\beta$-symmetric integrable
on $\left[a,b\right]$, then
\[
\left(  \int_{a}^{b}\left\vert f\left(  t\right)  +g\left(  t\right)
\right\vert ^{p}d_{\alpha,\beta}t\right)  ^{\frac{1}{p}}\leqslant\left(
\int_{a}^{b}\left\vert f\left(  t\right)  \right\vert ^{p}d_{\alpha,\beta
}t\right)  ^{\frac{1}{p}}+\left(  \int_{a}^{b}\left\vert g\left(  t\right)
\right\vert ^{p}d_{\alpha,\beta}t\right)  ^{\frac{1}{p}}.
\]
\end{theorem}

\begin{proof}
One has
\begin{multline*}
\int_{a}^{b}\left\vert f\left(  t\right)
+g\left(  t\right)  \right\vert^{p}d_{\alpha,\beta}t
=\int_{a}^{b}\left\vert f\left(  t\right)  +g\left(
t\right)  \right\vert ^{p-1}\left\vert f\left(  t\right)  +g\left(  t\right)
\right\vert d_{\alpha,\beta}t\\
\leqslant\int_{a}^{b}\left\vert f\left(  t\right)  \right\vert \left\vert
f\left(  t\right)  +g\left(  t\right)  \right\vert ^{p-1}d_{\alpha,\beta
}t+\int_{a}^{b}\left\vert g\left(  t\right)  \right\vert \left\vert f\left(
t\right)  +g\left(  t\right)  \right\vert ^{p-1}d_{\alpha,\beta}t.
\end{multline*}
Applying H\"{o}lder's inequality (Theorem~\ref{Holders Inequality}) with
$q=p/(p-1)$, we obtain
\begin{align*}
& \int_{a}^{b}\left\vert f\left(  t\right)  +g\left(  t\right)  \right\vert
^{p}d_{\alpha,\beta}t
\leqslant \left(  \int_{a}^{b}\left\vert f\left(
t\right)  \right\vert ^{p}d_{\alpha,\beta}t\right)  ^{\frac{1}{p}}\left(
\int_{a}^{b}\left\vert f\left(  t\right)  +g\left(  t\right)  \right\vert
^{\left(  p-1\right)  q}d_{\alpha,\beta}t\right)  ^{\frac{1}{q}}\\
& +\left(  \int_{a}^{b}\left\vert g\left(  t\right)  \right\vert ^{p}
d_{\alpha,\beta}t\right)  ^{\frac{1}{p}}\left(  \int_{a}^{b}\left\vert
f\left(  t\right)  +g\left(  t\right)  \right\vert ^{\left(  p-1\right)
q}d_{\alpha,\beta}t\right)  ^{\frac{1}{q}}\\
=& \left[  \left(  \int_{a}^{b}\left\vert f\left(  t\right)  \right\vert
^{p}d_{\alpha,\beta}t\right)  ^{\frac{1}{p}}+\left(  \int_{a}^{b}\left\vert
g\left(  t\right)  \right\vert ^{p}d_{\alpha,\beta}t\right)  ^{\frac{1}{p}
}\right]
\left(  \int_{a}^{b}\left\vert f\left(  t\right)  +g\left(
t\right)  \right\vert ^{\left(  p-1\right)  q}d_{\alpha,\beta}t\right)
^{\frac{1}{q}}.
\end{align*}
Therefore,
\begin{equation*}
\frac{\int_{a}^{b}\left\vert f\left(  t\right)
+g\left(  t\right)  \right\vert^{p}d_{\alpha,\beta}t}{\left(
\int_{a}^{b}\left\vert f\left(
t\right)  +g\left(  t\right)  \right\vert ^{\left(  p-1\right)  q}
d_{\alpha,\beta}t\right)^{\frac{1}{q}}}
\leqslant
\left(  \int_{a}^{b}\left\vert f\left(
t\right)  \right\vert ^{p}d_{\alpha,\beta}t\right)^{\frac{1}{p}}
+\left(\int_{a}^{b}\left\vert g\left(  t\right)  \right\vert ^{p}d_{\alpha,\beta
}t\right)  ^{\frac{1}{p}}.
\end{equation*}
\end{proof}

Our $\alpha,\beta$-symmetric calculus is more general
than the standard $h$-calculus. In particular, all our results give,
as corollaries, results in the classical quantum $h$-calculus
by choosing $\alpha=h>0$ and $\beta=0$.


\begin{acknowledgement}
Work supported by FEDER and Portuguese funds,
COMPETE reference FCOMP-01-0124-FEDER-022690,
and CIDMA and FCT, project PEst-C/MAT/UI4106/2011.
Brito da Cruz is also supported by FCT through
the Ph.D. fellowship SFRH/BD/33634/2009.
\end{acknowledgement}




\begin{thebibliography}{99}

\bibitem{Agarval}
R. Agarwal, M. Bohner\ and\ A. Peterson,
Inequalities on time scales: a survey,
Math. Inequal. Appl. {\bf 4} (2001), no.~4, 535--557.

\bibitem{Aldwoah}
K. A. Aldwoah,
\emph{Generalized time scales and associated difference equations},
Ph.D. thesis, Cairo University, 2009.

\bibitem{withMiguel01}
A. M. C. Brito da Cruz, N. Martins\ and\ D. F. M. Torres,
Higher-order Hahn's quantum variational calculus,
Nonlinear Anal. {\bf 75} (2012), no.~3, 1147--1157.
{\tt arXiv:1101.3653}

\bibitem{Kac}
V. Kac\ and\ P. Cheung,
{\it Quantum calculus},
Springer, New York, 2002.

\bibitem{MalinowskaTorres}
A. B. Malinowska\ and\ D. F. M. Torres,
The Hahn quantum variational calculus,
J. Optim. Theory Appl. {\bf 147} (2010), no.~3, 419--442.
{\tt arXiv:1006.3765}

\bibitem{Thomsom}
B. S. Thomson,
{\it Symmetric properties of real functions},
Dekker, New York, 1994.

\end{thebibliography}
\end{document}